\documentclass[a4paper,10pt,english]{smfart}
\usepackage[english]{babel}
\usepackage[OT2,T1]{fontenc}
\DeclareSymbolFont{cyrletters}{OT2}{wncyr}{m}{n}
\DeclareMathSymbol{\Sha}{\mathalpha}{cyrletters}{"58}
\usepackage{lmodern}
\usepackage{graphicx}
\usepackage{tabularx}
\usepackage{amsmath,amsthm,amssymb,amsfonts}
\usepackage[a4paper,left=4cm,right=4cm,top=4cm,bottom=4cm]{geometry}
\numberwithin{equation}{section}
\title[Height of rational points on congruent number elliptic curves]{Height of rational points on congruent number elliptic curves}
\author{Pierre Le Boudec}
\subjclass{$11$D$45$, $11$G$05$, $11$G$50$, $14$G$05$}
\address{Departement Mathematik und Informatik \\ Fachbereich Mathematik \\ Spiegelgasse $1$ \\ $4051$ Basel \\ Switzerland}
\email{pierre.leboudec@unibas.ch}

\begin{document}

\makeatletter
\def\imod#1{\allowbreak\mkern10mu({\operator@font mod}\,\,#1)}
\makeatother

\newtheorem{lemma}{Lemma}
\newtheorem{theorem}{Theorem}
\newtheorem{corollary}{Corollary}
\newtheorem{proposition}{Proposition}
\newtheorem{conjecture}{Conjecture}
\newtheorem{conj}{Conjecture}
\renewcommand{\theconj}{\Alph{conj}}
\newtheorem{theo}{Theorem}
\renewcommand{\thetheo}{\Alph{theo}}

\setcounter{theorem}{-1}
\setcounter{conjecture}{-1}

\newcommand{\vol}{\operatorname{vol}}
\newcommand{\D}{\mathrm{d}}
\newcommand{\rank}{\operatorname{rank}}
\newcommand{\Pic}{\operatorname{Pic}}
\newcommand{\Gal}{\operatorname{Gal}}
\newcommand{\meas}{\operatorname{meas}}
\newcommand{\Spec}{\operatorname{Spec}}
\newcommand{\eff}{\operatorname{eff}}
\newcommand{\rad}{\operatorname{rad}}
\newcommand{\sq}{\operatorname{sq}}
\newcommand{\tors}{\operatorname{tors}}
\newcommand{\Cl}{\operatorname{Cl}}
\newcommand{\Reg}{\operatorname{Reg}}
\newcommand{\Sel}{\operatorname{Sel}}
\newcommand{\corank}{\operatorname{corank}}
\newcommand{\an}{\operatorname{an}}

\begin{abstract}
We prove that a positive proportion of squarefree integers are congruent numbers such that the canonical height of the lowest non-torsion rational point on the corresponding elliptic curve satisfies a strong lower bound.
\end{abstract}

\maketitle

\tableofcontents

\section{Introduction}

Let $(A,B) \in \mathbb{Z}^2$ be fixed and such that $4 A^3 + 27 B^2 \neq 0$. For every squarefree integer $d \geq 1$, we let $E_d$ be the elliptic curve defined over $\mathbb{Q}$ by the equation
\begin{equation*}
d y^2 = x^3 + A x + B.
\end{equation*}
We recall that the Mordell-Weil Theorem asserts that $E_d(\mathbb{Q})$ is a finitely generated abelian group. We respectively let
$\rank E_d(\mathbb{Q})$ and $E_d(\mathbb{Q})_{\tors}$ denote its rank and its torsion subgroup. We also let $\hat{h}_{E_d}$ be the canonical height on $E_d$. The author \cite{MR3455753} has recently investigated the distribution, as $d$ varies, of the quantity
$\eta_d(A,B)$ defined by
\begin{equation*}
\log \eta_d(A,B) = \min \{ \hat{h}_{E_d}(P), P \in E_d(\mathbb{Q}) \smallsetminus E_d(\mathbb{Q})_{\tors} \},
\end{equation*}
if $\rank E_d(\mathbb{Q}) \geq 1$ and $\eta_d(A,B) = \infty$ if $\rank E_d(\mathbb{Q}) = 0$.

The goal of this note is to push this investigation further in the case of the congruent number elliptic curves. We recall that a congruent number is a positive integer which is the area of a right triangle with rational side lengths. A classical result states that a squarefree integer
$d \geq 1$ is a congruent number if and only if the elliptic curve defined over $\mathbb{Q}$ by the equation $d y^2 = x^3-x$ has positive rank. It is worth noting that the Parity Conjecture (which asserts that the analytic and algebraic ranks of an elliptic curve should have same parity) implies that any squarefree integer congruent to $5$, $6$ or $7$ modulo $8$ is a congruent number.

We start by summarizing the work of the author \cite{MR3455753}. In order to do so, we recall some classical vocabulary in analytic number theory. Given a set $S \subset \mathbb{Z}_{\geq 1}$ and a property $\mathfrak{P}$ which takes as argument a positive integer, we say that the property $\mathfrak{P}$ holds for almost every element of $S$ if
\begin{equation*}
\lim_{X \to \infty} \frac{\# \{ n \in S, n \leq X, \mathfrak{P}(n) \} }{\# \{ n \in S, n \leq X \}} = 1.
\end{equation*}
Also, we say that a positive proportion of elements of $S$ satisfy the property $\mathfrak{P}$ if
\begin{equation*}
\liminf_{X \to \infty} \frac{\# \{ n \in S, n \leq X, \mathfrak{P}(n) \} }{\# \{ n \in S, n \leq X \}} > 0.
\end{equation*}

Goldfeld's Conjecture \cite{MR564926} implies in particular that a positive proportion of squarefree integers $d \geq 1$ should satisfy
$\rank E_d(\mathbb{Q}) \geq 1$. As a result, the conjecture of the author \cite[Conjecture $1$]{MR3455753} is expected to be equivalent to the following.

\begin{conjecture}
Let $(A,B) \in \mathbb{Z}^2$ be such that $4 A^3 + 27 B^2 \neq 0$, and let $\varepsilon > 0$ be fixed. For almost every squarefree integer $d \geq 1$ such that $\rank E_d(\mathbb{Q}) \geq 1$, we have the lower bound
\begin{equation*}
\eta_d(A,B) > e^{d^{1/2 - \varepsilon}}.
\end{equation*}
\end{conjecture}

This conjecture is unfortunately out of reach. Nevertheless, in \cite{MR3455753} the author made a first step (particularly significant in the case of the congruent number elliptic curves) in this direction. More precisely, combining the proof of \cite[Theorem $2$]{MR3455753} and the work of Heegner \cite{MR0053135}, we see that \cite[Theorem~$2$]{MR3455753} can be restated as follows.

\begin{theorem}
Let $\varepsilon > 0$ be fixed. For almost every squarefree congruent number $d \geq 1$, we have the lower bound
\begin{equation*}
\eta_d(-1,0) > d^{5/8 - \varepsilon}.
\end{equation*}
\end{theorem}

We note that \cite[Theorem $1$]{MR3455753} may also be restated in an analogous way using the result of Perelli and Pomyka\l a
\cite[Theorem~$1$]{MR1450922} and the work of Gross and Zagier \cite{MR833192}.

The goal of this note is to establish the following result.

\begin{theorem}
\label{Theorem}
A positive proportion of squarefree integers $d \geq 1$ are congruent numbers which satisfy the lower bound
\begin{equation*}
\eta_d(-1,0) > d^{0.845}.
\end{equation*}
\end{theorem}

Let us describe the ingredients of the proof of Theorem \ref{Theorem}. The first step consists in restricting our attention to squarefree integers which have a large prime factor. We note that a similar construction was exploited by Fouvry and Jouve \cite{MR3069056} in their investigation of the size of the fundamental solution of the Pell equation. Then, we observe that it is possible to take advantage of this property by parametrizing rational points using a complete $2$-descent process as in the previous work of the author \cite{MR3455753}. To complete the proof, we finally make use of the recent result of Smith \cite[Theorem $1$.$5$]{Smith} which states that a positive proportion of squarefree integers $d \geq 1$ are congruent numbers.

It is worth noting that if we assume that for some $a \in \{ 5, 6, 7 \}$, almost every squarefree integer $d \geq 1$ congruent to $a$ modulo $8$ is a congruent number, then the exponent $0.845$ appearing in Theorem \ref{Theorem} can be replaced by
$1 - \varepsilon$ for any fixed $\varepsilon > 0$.

We finish this introduction by mentioning that it follows from our method and the work of Heegner \cite{MR0053135} that for almost every prime congruent number $p$, we have the lower bound
\begin{equation*}
\eta_p(-1,0) > p^{1 - \varepsilon},
\end{equation*}
for any fixed $\varepsilon > 0$.

\subsection*{Acknowledgements}

It is a great pleasure for the author to thank Fabien Pazuki for interesting conversations related to the topics of this article.

This work was initiated while the author was working as an Instructor at the \'{E}cole Polytechnique F\'{e}d\'{e}rale de Lausanne. The financial support and the wonderful working conditions that the author enjoyed during the four years he worked at this institution are gratefully acknowledged.

The research of the author is integrally funded by the Swiss National Science Foundation through the SNSF Professorship number $170565$ awarded to the project \textit{Height of rational points on algebraic varieties}. Both the financial support of the SNSF and the perfect working conditions provided by the University of Basel are gratefully acknowledged.

\section{Preliminaries}

We let $\mathcal{S}(X)$ denote the set of positive squarefree integers up to $X$ and we let $\mathcal{P}$ be the set of prime numbers. For $\vartheta \in (0,1/2)$, we introduce the set
\begin{equation*}
\mathcal{T}_{\vartheta}(X) = \left\{m p \in ( X^{2 \vartheta}, X ],
\begin{array}{l l}
(m,p) \in \mathcal{S}(X^{\vartheta}) \times \mathcal{P} \\
m p = 5 \imod{8}
\end{array}
\right\}.
\end{equation*}
It is important to note that $\mathcal{T}_{\vartheta}(X) \subset \mathcal{S}(X)$. Indeed, if $m p > X^{2 \vartheta}$ and
$m \leq X^{\vartheta}$ then $p > X^{\vartheta}$ and thus $p \nmid m$.

A crucial ingredient in the proof of Theorem \ref{Theorem} is the lower bound $\# \mathcal{T}_{\vartheta}(X) \gg X$. More precisely, we prove the following lemma which gives an asymptotic formula for the cardinality of the set $\mathcal{T}_{\vartheta}(X)$.

\begin{lemma}
\label{Lemma T}
Let $\vartheta \in (0,1/2)$ be fixed. We have the estimate
\begin{equation*}
\# \mathcal{T}_{\vartheta}(X) = - \frac{\log ( 1 - \vartheta) }{\pi^2} X + O \left( \frac{X}{\log X} \right).
\end{equation*}
\end{lemma}

\begin{proof}
We start by proving that if $m_1, m_2 \in \mathcal{S}(X^{\vartheta})$ and $p_1, p_2 \in \mathcal{P}$ are such that $m_1 p_1 = m_2 p_2$ and $m_1 p_1 > X^{2 \vartheta}$ then $(m_1,p_1) = (m_2,p_2)$. Indeed, if we assume that $p_1 \neq p_2$ then we must have
$p_1 \mid m_2$, but this is impossible since $m_2 \leq X^{\vartheta}$ and, as already noticed, $p_1 > X^{\vartheta}$. So $p_1 = p_2$ and thus also $m_1 = m_2$.

It follows from this observation that
\begin{equation*}
\# \mathcal{T}_{\vartheta}(X) =
\sum_{p \leq X} \sum_{\substack{m \leq X^{\vartheta} \\ X^{2 \vartheta} < m p \leq X \\ mp = 5 \imod{8}}} |\mu(m)|.
\end{equation*}
Moreover, we note that
\begin{equation*}
\sum_{p \leq X^{1 - \vartheta}} \sum_{m \leq X^{\vartheta}} |\mu(m)| \ll \pi(X^{1-\vartheta}) X^{\vartheta},
\end{equation*}
where $\pi(X^{1-\vartheta})$ denotes the number of prime numbers up to $X^{1-\vartheta}$. We also note that
\begin{equation*}
\sum_{p \leq X} \sum_{\substack{m \leq X^{\vartheta} \\ m p \leq X^{2 \vartheta}}} |\mu(m)| \ll
X^{2 \vartheta} \log \log X.
\end{equation*}
In addition, if $p > X^{1 - \vartheta}$ and $m p \leq X$ then we necessarily have $m \leq X^{\vartheta}$. Therefore, Chebyshev's upper bound and the fact that $\vartheta < 1/2$ imply that
\begin{equation*}
\# \mathcal{T}_{\vartheta}(X) =
\sum_{X^{1 - \vartheta} < p \leq X} \sum_{\substack{m \leq X/p \\ mp = 5 \imod{8}}} |\mu(m)| +
O \left( \frac{X}{\log X} \right).
\end{equation*}
For any odd prime $p$, we have the classical asymptotic formula
\begin{equation*}
\sum_{\substack{m \leq X/p \\ mp = 5 \imod{8}}} |\mu(m)| = \frac1{\pi^2} \frac{X}{p} + O \left( \frac{X^{1/2}}{p^{1/2}} \right).
\end{equation*}
We thus deduce
\begin{equation*}
\# \mathcal{T}_{\vartheta}(X) = \frac{X}{\pi^2} \sum_{X^{1 - \vartheta} < p \leq X} \frac1{p} + O \left( \frac{X}{\log X} \right).
\end{equation*}
Finally, it follows from the Second Theorem of Mertens that
\begin{equation*}
\sum_{X^{1 - \vartheta} < p \leq X} \frac1{p} = - \log ( 1 - \vartheta ) + O  \left( \frac1{\log X} \right),
\end{equation*}
which completes the proof.
\end{proof}

For $\alpha > 0$, we define
\begin{equation*}
\mathcal{N}_{\alpha, \vartheta}(X) = \# \{ d \in \mathcal{T}_{\vartheta}(X), \eta_d( -1, 0) \leq d^{1/8 + \alpha} \}.
\end{equation*}
The following lemma gives an upper bound for $\mathcal{N}_{\alpha, \vartheta}(X)$.

\begin{lemma}
\label{Lemma E}
Let $\alpha, \varepsilon > 0$ and $\vartheta \in (0,1/2)$ be fixed. We have the upper bound
\begin{equation*}
\mathcal{N}_{\alpha, \vartheta}(X) \ll X^{1/8 + \alpha + \vartheta/2 + \varepsilon}.
\end{equation*}
\end{lemma}

\begin{proof}
We proceed as in the proof of \cite[Lemma $6$]{MR3455753}. More precisely, we first use \cite[Lemma $3$]{MR3455753} to compare the canonical height and the Weil height and we then use \cite[Lemma $2$]{MR3455753} to parametrize the rational points using a complete $2$-descent process. We obtain
\begin{equation*}
\mathcal{N}_{\alpha, \vartheta}(X) \leq 2
\# \left\{ (\nu, \mathbf{d}, \mathbf{b}) \in \{ -1, 1 \} \times \mathbb{Z}_{\geq 1}^4 \times \mathbb{Z}_{\geq 1}^4,
\begin{array}{l}
d_1 d_2 d_3 d_4 \in \mathcal{T}_{\vartheta}(X) \\
\gcd(d_1 b_1, d_2 b_2) = 1 \\
d_2 b_2^2 - \nu d_1 b_1^2 = d_3 b_3^2 \\
\nu d_2 b_2^2 + d_1 b_1^2 = d_4 b_4^2 \\
d_1 b_1^2, d_2 b_2^2 \ll X^{1/4 + 2 \alpha}
\end{array}
\right\},
\end{equation*}
where we have set $\mathbf{d} = (d_1, d_2, d_3, d_4)$ and $\mathbf{b} = (b_1, b_2, b_3, b_4)$. Since we have
\begin{equation*}
\mathcal{T}_{\vartheta}(X) \subset \{ m p \in \mathcal{S}(X), m \leq X^{\vartheta}, p \in \mathcal{P} \},
\end{equation*}
we see that if $d_1 d_2 d_3 d_4 \in \mathcal{T}_{\vartheta}(X)$ then
\begin{equation*}
\frac{d_1 d_2 d_3 d_4}{d_{i_0}} \leq X^{\vartheta},
\end{equation*}
for some $i_0 \in \{ 1, 2, 3, 4 \}$. Therefore, we have
\begin{equation*}
\mathcal{N}_{\alpha, \vartheta}(X) \ll \max_{\substack{i_0 \in \{ 1, 2, 3, 4 \} \\ \nu \in \{ -1, 1 \}}}
\# \left\{ (\mathbf{d}, \mathbf{b}) \in \mathbb{Z}_{\geq 1}^4 \times \mathbb{Z}_{\geq 1}^4,
\begin{array}{l}
|\mu(d_1 d_2 d_3 d_4)| = 1 \\
\gcd(d_1 b_1, d_2 b_2) = 1 \\
d_2 b_2^2 - \nu d_1 b_1^2 = d_3 b_3^2 \\
\nu d_2 b_2^2 + d_1 b_1^2 = d_4 b_4^2 \\
d_1 d_2 d_3 d_4/d_{i_0} \leq X^{\vartheta} \\
d_1 b_1^2, d_2 b_2^2 \ll X^{1/4 + 2 \alpha}
\end{array}
\right\}.
\end{equation*}
Reasoning exactly as in the proof of \cite[Lemma $6$]{MR3455753}, we obtain
\begin{equation*}
\mathcal{N}_{\alpha, \vartheta}(X) \ll X^{\varepsilon} \max_{i_0 \in \{ 1, 2, 3, 4 \}} \sum_{\substack{D_i, B_i \\ i \in \{ 1, 2, 3, 4 \}}} \left( \frac{D_1 D_2 D_3 D_4}{D_{i_0}} \right)^{2/3} \left( \frac{B_1 B_2 B_3 B_4}{B_{i_0}} \right)^{1/3},
\end{equation*}
where the sum is over the $D_i \geq 1/2$, $B_i \geq 1/2$, $i \in \{ 1, 2, 3, 4 \}$, running over the set of powers of $2$ and satisfying
\begin{equation}
\label{Condition 1}
\frac{D_1 D_2 D_3 D_4}{D_{i_0}} \leq X^{\vartheta},
\end{equation}
and
\begin{equation}
\label{Condition 2}
D_i B_i^2 \ll X^{1/4 + 2 \alpha},
\end{equation}
for $i \in \{ 1, 2, 3, 4 \}$. Using the upper bounds \eqref{Condition 2}, we thus get
\begin{equation*}
\mathcal{N}_{\alpha, \vartheta}(X) \ll X^{1/8 + \alpha + \varepsilon} \max_{i_0 \in \{ 1, 2, 3, 4 \}}
\sum_{\substack{D_i, B_i \\ i \in \{ 1, 2, 3, 4 \}}} \left( \frac{D_1 D_2 D_3 D_4}{D_{i_0}} \right)^{1/2}.
\end{equation*}
Using the upper bound \eqref{Condition 1}, we eventually deduce
\begin{equation*}
\mathcal{N}_{\alpha, \vartheta}(X) \ll X^{1/8 + \alpha + \vartheta/2 + 2 \varepsilon},
\end{equation*}
which completes the proof.
\end{proof}

Combining Lemmas \ref{Lemma T} and \ref{Lemma E}, we immediately obtain the following result.

\begin{lemma}
\label{Lemma T2}
Let $\alpha > 0$ and $\vartheta \in (0, 1/2)$ be fixed and such that $\alpha + \vartheta/2 < 7/8$. We have the estimate
\begin{equation*}
\# \{ d \in \mathcal{T}_{\vartheta}(X), \eta_d( -1, 0) > d^{1/8+\alpha} \} =
- \frac{\log ( 1 - \vartheta) }{\pi^2} X + O \left( \frac{X}{\log X} \right).
\end{equation*}
\end{lemma}

We let $\mathcal{C}(X)$ denote the set of squarefree congruent numbers up to $X$. The following result was recently established by Smith
(see \cite[Theorem $1$.$5$]{Smith}) and is crucial in the proof of Theorem~\ref{Theorem}.

\begin{lemma}
\label{Lemma C}
We have the lower bound
\begin{equation*}
\liminf_{X \to \infty} \frac{\# \{ d \in \mathcal{C}(X), d = 5 \imod{8} \}}{\# \{ d \in \mathcal{S}(X), d = 5 \imod{8} \} } \geq 0.629.
\end{equation*}
\end{lemma}

\section{Proof of Theorem \ref{Theorem}}

Our goal is to prove that
\begin{equation}
\label{Goal}
\liminf_{X \to \infty}
\frac{\# \{ d \in \mathcal{C}(X), \eta_d( -1, 0) > d^{0.845} \}}{\# \mathcal{S}(X)} > 0.
\end{equation}
Recall that we have the classical asymptotic formula
\begin{equation*}
\# \{ d \in \mathcal{S}(X), d = 5 \imod{8} \} = \frac{X}{\pi^2} + O( X^{1/2} ).
\end{equation*}
As a result, choosing $\vartheta = 0.30996$ and $\alpha = 0.72$ in Lemma \ref{Lemma T2}, we get
\begin{equation}
\label{AF}
\frac{\# \{ d \in \mathcal{T}_{\vartheta}(X), \eta_d( -1, 0) > d^{0.845} \}}{\# \{ d \in \mathcal{S}(X), d = 5 \imod{8} \}} =
- \log ( 1 - \vartheta)  + O \left( \frac1{\log X} \right).
\end{equation}
Since $- \log (1 - \vartheta) + 0.629 > 1$, putting together the estimate \eqref{AF} and Lemma \ref{Lemma C}, we deduce
\begin{equation*}
\liminf_{X \to \infty}
\frac{\# \{ d \in \mathcal{T}_{\vartheta}(X), \eta_d( -1, 0) > d^{0.845} \} + \# \{ d \in \mathcal{C}(X), d = 5 \imod{8} \}}
{\# \{ d \in \mathcal{S}(X), d = 5 \imod{8} \}} > 1.
\end{equation*}
This implies the lower bound \eqref{Goal} and thus completes the proof of Theorem~\ref{Theorem}.

\bibliographystyle{amsalpha}
\bibliography{biblio}

\end{document}